\documentclass{tsp-short}

\newtheorem{thm}{Theorem}[section]
\newtheorem{lem}{Lemma}[section]

\theoremstyle{remark}

\theoremstyle{definition}

\usepackage[cp1251]{inputenc}
\usepackage{amsmath, amsthm, amsfonts, amssymb}

\newcommand{\vect}{\overrightarrow}
\newcommand{\mfX}{\mathfrak X}
\newcommand{\mcX}{\mathcal X}
\newcommand{\ov}{\overline}

\newcommand{\Id}{\mathop{\rm Id}}

\newcommand{\vk}{\varkappa}
\newcommand{\ve}{\varepsilon}
\newcommand{\cF}{{\mathcal F}}

\newcommand{\cE}{{\mathcal E}}
\newcommand{\cN}{{\mathcal N}}

\newcommand{\cG}{{\mathcal G}}

\newcommand{\mfe}{\mathfrak e}

\newcommand{\wt}{\widetilde}

\newcommand{\1}{1\!\!\,{\rm I}}
\newcommand{\mbR}{{\mathbb R}}

\newcommand{\mbN}{{\mathbb N}}
\newcommand{\e}{\mathrm{e}}

\DeclareMathOperator\E{E}

\DeclareMathOperator\Prob{P}

\theoremstyle{plain}
\theoremstyle{definition}

\theoremstyle{remark}

\makeatletter
\@namedef{subjclassname@2020}{%
  \textup{2020} Mathematics Subject Classification}
\makeatother
\begin{document}

\title
[Representations of the finite-dimensional point densities]
{Representations of the finite-dimensional point densities in Arratia flows with drift}

\author{A.A.Dorogovtsev}
\address{A.A.Dorogovtsev: Institute of Mathematics, National Academy of Sciences of Ukraine, Tereshchenkivska Str. 3, Kiev 01601, Ukraine; National Technical University of Ukraine ``Igor Sikorsky Kyiv Polytechnic Institute'', Institute of Physics and Technology, Peremogi avenue 37, Kiev 01601, Ukraine}
\email{andrey.dorogovtsev@gmail.com}

\author[N.B.Vovchanskii]{N.B.Vovchanskii}
\address{N.B.Vovchanskii: Institute of Mathematics, National Academy of Sciences of  Ukraine, Tereshchenkivska Str. 3, Kiev 01601, Ukraine}
\email{vovchansky.m@gmail.com}

\begin{abstract}
We derive representations for  finite-dimensional densities of the point process associated with an Arratia flow with drift in terms of conditional expectations of the stochastic exponentials appearing in the analog of the Girsanov theorem for the Arratia flow.  
\end{abstract}

\maketitle
\section{Introduction}
The study of the point process associated with an Arratia flow $\{X^a(u, t)| u\in[0; 1], t\in [0; T]\}$ with drift $a$ \cite[Section 7]{1} is carried out in the present paper by means of special $(n, k)$-point densities $p^{a,n,k}_t, k\leq n.$ Such densities constitute a generalization of those discussed in \cite{2, 3, 9} and are informally defined via the formula
\begin{align*}
\Prob &\Big(\forall \ i=\ov{1, n} \ X^a(u_i, t)\in \bigcup^k_{j=1}[y_j; y_j+dy_j],\ \forall \ j=\ov{1, k}
  \\
&  \phantom{aaa} \{X^a(u_l,t)\mid l=\ov{1, n}\}\cap [y_j; y_j+dy_j]\ne \varnothing\Big)=
p^{a, n, k}_t(u; y)dy_1\ldots dy_k,
\end{align*}
the strict definition to be provided later in the text. 

We find the Radon--Nikodym representation for $p^{a,n,k}_t$ in terms of $p^{0,n,k}_t.$ It is known \cite[p. 194]{1} that 
that the distribution of the random process $(X^a(u_1,\cdot), \ldots, X^a(u_n, \cdot))$ is absolutely continuous in $(C([0; T]))^n$ w.r.t. the distribution of $(X^0(u_1,\cdot), \ldots, X^0(u_n, \cdot))$ with density
\begin{equation}
\label{eq1}
\wt\cE^a_{T,n}(u) = \exp\left\{\sum^n_{k=1}\int^{\tau_k}_0a(X^0(u_k, t))dX^0(u_k, t)-\frac{1}{2} \sum^n_{k=1}\int^{\tau_k}_0a^2(X^0(u_k, t))dt\right\},
\end{equation}
where $\tau_1=T$ and  
\begin{align*}
\tau_k&=\inf\Bigg\{t\mid \prod^{k-1}_{j=1}\left(X^0(u_j, t)-X^0(u_k,t)\right)=0\Bigg\} \wedge T,\quad  k=\ov{1,n},
\end{align*}
where  $\inf\emptyset = +\infty$ by definition.
Moreover, the distribution of an Arratia flow with a bounded Lipschitz continuous drift $a$ as a random element in $D([0; 1], C([0; T]))$ is absolutely continuous  w.r.t. the distribution of the Arratia flow with zero drift \cite[Theorem 7.3.1]{1}. 

Since the definition of the densities $p^{a,n,k}_t$ contains the condition for the flow to hit the neighborhoods of certain points at time $t,$ we firstly investigate the distribution of \eqref{eq1} conditional on $(X^0(u_1, T),\ldots, X^0(u_n,T)).$

Hereinafter the superscript $a=0$ is dropped in the case of zero drift, and  $a$ is always assumed to be bounded and Lipschitz continuous. We write $x=(x_1, \ldots, x_n)$ for points in $\mbR^n, n\in\mbN.$

\section{On Brownian bridges and related conditional distributions}
\label{section1}
 Put $\Delta_n = \{u\in\mbR^n\mid u_1 < \ldots < u_n\},$ $n\in\mbN.$ The following constructive scheme is used. Assume $W=(w_1, \ldots, w_n)$ to be a standard Wiener process in $\mbR^n$ started at 0. Put  $\wt{w}_1=w_1, \theta_1(u)=T$ and define 
\begin{align*}
\theta_k(u) &=\inf\left\{T; t\mid \wt{w}_{k-1}(t)+u_{k-1}=w_k(t)+u_k\right\}, \\
\wt{w}_k(t)&=-u_k+(u_k + w_k(t))\1(t<\theta_k(u))+(u_{k-1}+\wt{w}_{k-1}(t))\1(t\geq\theta_k(u)), \quad k=\ov{2,n},
\end{align*}
where $u\in\Delta_n.$ Denote $\wt{W}=(\wt{w}_1, \ldots,\wt{w}_n).$ Then one can easily verify the following statement.
\begin{lem}
\label{lem1}
In $\left(C([0; T])\right)^n$
\[
\left(X(u_1, \cdot), \ldots, X(u_n, \cdot)\right)\stackrel{d}{=}u+\wt{W},
\]
and the expression $\wt\cE^a_{T,n}(u)$ in \eqref{eq1} has the same distribution as
\begin{align*}
\cE^a_{T,n}(W, u)&=\exp\left\{\sum^n_{k=1}\int^{\theta_k(u)}_0a(u_k+w_k(t))dw_k(t) -\frac{1}{2}\sum^n_{k=1}\int^{\theta_k(u)}_0 a^2(u_k+w_k(t))dt\right\}.
\end{align*}
\end{lem}
The It\^o stochastic integrals that participate in the definition of $\cE^a_T(W, u)$ can be expressed in terms of stochastic integrals w.r.t. the Brownian bridges $\eta=(\eta_1, \ldots, \eta_n)$ associated with $W$:
\begin{equation}
\label{eq2}
w_k(t)=\frac{t}{T}w_k(T)-\eta_k(t), \quad t\in[0; T],  k=\ov{1,n}.
\end{equation}

We refer  to \cite[\S 5.6.B]{4}\cite[pp. 299-300]{protter.sde} for the general exposition of the theory of  Brownian bridges.

Define the filtration
\[
\cF_t=\sigma\left(\eta_k(s), s\leq t, k=\ov{1, n}\right), \quad t\in[0; T],
\]
which is further supposed to be augmented in a standard way. Each $\eta_k$ is the solution of  the SDE
\begin{align}
\label{eq3}
d\eta_k(t)&=d\beta_k(t)-\frac{\eta_k(t)}{T-t}dt, \quad t\in[0; T), \nonumber \\
\eta_k(0)&=0,
\end{align}
where every $\beta_k$ is a  $(\cF_t)_{t\in[0; T]}-$Wiener process and $\eta_k(T)=0.$ At the same time, every $\eta_k$ admits the representation
\begin{equation}
\label{eq4}
\eta_k(t)=(T-t)b_k\left(\frac{t}{T(T-t)}\right), \quad t\in[0; T),
\end{equation}
with $b_1, \ldots, b_n$ being independent standard  Wiener processes.

Since
\begin{gather}
\label{eq:bb.semimart}
\E \int^T_0\frac{|\eta_1(t)|}{T-t}dt\le \int^T_0\frac{(\E\eta^2_1(t))^{1/2}}{T-t}dt\leq \int^T_0\left(\frac{t}{(T-t)T}\right)^{1/2}dt<+\infty,
\end{gather}
each Brownian bridge $\eta_k$ is a semimartingale w.r.t its own filtration and the stochastic integrals w.r.t. the coordinates of $\eta$ are defined in a usual way, at least for bounded progressively measurable integrands (see \cite{5} for the full characterization of possible integrands).   

For any $u\in\Delta^n, y\in\mbR^n$ define the following random process $\eta^{u,y}=(\eta^{u,y}_1,\ldots, \eta^{u,y}_n)$ in $\mbR^n:$ 
\[
\eta^{u,y}(t) = \eta(t)+\left(1 -\frac{t}{T}\right)u+\frac{t}{T}y, \quad t\in[0;T],
\] 
and put  
\begin{align*}
\theta_{ij}(u) &=\inf\left\{t \mid w_i(t)+u_i=w_j(t)+u_j\right\}\wedge T, \\ 
\theta_{ij}(u,y)&=\inf\left\{t\mid \eta_i^{u,y}(t)=\eta_j^{u,y}(t)\right\}\wedge T, \\
& j=\ov{1, i-1}, i=\ov{2,n}, y\in\mbR^n.
\end{align*}
Additionally, put 
$$
\theta_{kk}(u)=\theta_{jj}(u,y)=T, \quad k,j=\ov{1, n}.
$$ 
Then a.s. 
$$\theta_{ij}(u)=\theta_{ij}(u,W(T)+u), \quad j=\ov{1, i}, i=\ov{1,n}.
$$
\begin{lem}
\label{lem3}
For all $u\in\Delta_n$ with probability $1$ all $\theta_{ij}(u)$ that are less than $T$ are distinct. For all $u\in\Delta_n$ and $y\in\mbR^n$ with probability $1$ all $\theta_{ij}(u,y)$ that are less than $T$ are distinct.
\end{lem}
\begin{proof}
The first assertion is trivial. The second one can be deduced from \eqref{eq4}.
\end{proof}

To describe sequences of collisions in finite-dimensional motions of an Arratia flow, we use the following construction that was introduced in \cite[Definition 1.2]{DorVovApproximations} and is a reformulation of the one presented in \cite[pp. 433-434]{dorogovtsev.2011.multiplicative}. Put
\begin{align*}
Sh_{n, k}&=\left\{(j_1, \ldots, j_k) \mid j_i\in\{1, \ldots, n-i\right\}, i=\ov{1, k}\}, \quad k=\ov{1, n}, \\
Sh_n&= \varnothing\vee\bigcup_{k=\ov{1, n-1}}Sh_{n,k}, \quad n\in\mbN.
\end{align*}
Recall $\wt{W}+u=(\wt{w}_1+u_1, \ldots, \wt{w}_n+u_n )$ to be coalescing Wiener processes constructed from the process $W+u.$ Let $n-\vk$ be the number of distinct values in the sequence $\{\wt{w}_i(T)+u_i\mid i=\ov{1,n}\}, \ \vk$ ranging in $\{0, \ldots, n-1\}.$ Let $\tau_1<\tau_2<\ldots<\tau_\vk$ be random moments such that
$$
\{\tau_1, \ldots, \tau_\vk\}=\left\{\theta_k(u)\mid \theta_k(u)<T, k=\ov{1, n}\right\}.
$$
By virtue of Lemma \ref{lem3}, such $\tau_1, \ldots, \tau_\vk$ exist a.s.. Put $j_1=\min\{i\mid  \exists j\ne i \ \wt{w}_j(\tau_1)+u_j=\wt{w}_i(\tau_1)+u_i\}$ and define the process $\wt{W}^{n-1}$  by excluding the $j_1$-th coordinate from the vector $\wt{w}+u.$ Then put $j_2=\min\{i\mid \exists \ j\ne i \ \wt{w}^{n-1}_j(\tau_2)=\wt{w}^{n-1}_i(\tau_2)\}, $  define $\wt{W}^{n-2}$ by excluding the $j_2$-th coordinate from the process $\wt{W}^{n-1}$ and repeat the procedure until a random collection $S(W+u)=(j_1, \ldots, j_\vk)\in Sh_{n, \vk}$ appears. We will call $S(W+u)$ the coalescing scheme for the process $W+u.$

Using Lemma \ref{lem3} one can find random numbers 
$\Lambda_{ij}, i=1,2, j=\ov{1, n}, \Lambda_{1p}\in\{1,\ldots, p\},$  $\Lambda_{2p}\in\{1, \ldots, \Lambda_{1p}\}, p=\ov{1, n},$ 
such that with probability $1$ 
\[
\theta_k(u) \1\left(S(W+u) = s\right)=\theta_{\Lambda_{1k}\Lambda_{2k}}(u) \1\left(S(W+u) = s\right), \quad k=\ov{1, n}.
\]
Moreover, there exist nonrandom numbers $\{\lambda_{ij}(s)\mid  i=1,2, j=\ov{1, n}\}$ such that a.s.
\[
\Lambda_{ij} = \sum_{s\in Sh_n} \lambda_{ij}(s) \1\left(S(W+u) = s\right), \quad i=1,2, j=\ov{1, n}.
\]
The collection $\{\lambda_{ij}(s)\mid  i=1,2, j=\ov{1, n}\}$ is completely determined by the value of the coalescence scheme $s$ and can be restored from the latter directly, though, due to cumbersomeness of the corresponding relations, we omit giving an explicit representation. 

For ease of the further presentation, put
\begin{align*}
a_k(t, u, y, s)& =\1(t\le \theta_{\lambda_{1k}(s)\lambda_{2k}(s)}(u,y))
\cdot a\left(\eta_k^{u,y}(t)\right), \\
& t\in[0; T],\ k=\ov{1, n},\ y\in\mbR^n,\ s\in Sh_n; \\
\mfe^a_{T,n}(u,y, s)&=\exp\Bigg\{\sum^n_{k=1}\int^T_0a_k(t, y, s)d\beta_k(t)+ \\
& \phantom{aaaa} +\sum^n_{k=1}\int^T_0a_k(t, u,y, s)\left(\frac{y_k-u_k}{T}-\frac{\eta_k(t)}{T-t}-\frac{1}{2}a_k(t,u, y, s)\right)ds
\Bigg
\}, \\
& y\in\mbR^n,\ s\in Sh_n.
\end{align*}
\begin{lem}
\label{lem4}
$\forall \ C>0 \ \forall \ k=\ov{1, n}$
$$
\E\e^{C\int^T_0\frac{|\eta_k(t)|}{T-t}dt}<+\infty.
$$
\end{lem}
\begin{proof}
By \eqref{eq:bb.semimart} the process $t\mapsto \frac{\eta_k(t)}{T-t}$ is a Gaussian random element in $L_1([0; T]),$ therefore the claim follows from the Fernique theorem \cite[Theorem 3.1]{6}.
\end{proof}
\begin{lem}
\label{lem5}
$\forall u\in\Delta_n \ \forall \ y\in\mbR^n \ \forall \ s\in Sh_n \ \forall p\ge 0$
$$
\E\left(\mfe^a_{T,n}(u,y, s)\right)^p\leq C_1 e^{C_2\|y\|},
$$
where 
\begin{align*}
C_1&=\e^{np|2p-1|T\cdot\|a\|^2_{L_\infty(\mbR)}+pn^{1/2}{\|u\|\cdot\|a\|_{L_\infty(\mbR)}}}\left(\E\e^{2p\|a\|_{L_\infty(\mbR)}\int^T_0\frac{|\eta_1(t)|}{T-t}dt}\right)^{n/2}, \\
C_2&=pn^{1/2} \|a\|_{L_\infty(\mbR)} .
\end{align*} 
\end{lem}
\begin{proof}
By the Cauchy inequality,
\begin{align*}
&\left(\E\left(\mfe^a_{T,n}(u,y, s)\right)^p\right)^2\leq \\
& \hspace{20mm} \leq \left[ \E\exp\left\{\sum^n_{k=1}\int^T_0 2pa_k(t,y,s)d\beta_k(t)-\frac{1}{2}\sum^n_{k=1}\int^T_0\left(2pa_k(t,y,s)\right)^2dt\right\} \right]\times
\\
& \hspace{25mm} \times 
 \Bigg[ \E\exp\Bigg\{ \sum^n_{k=1}\int^T_0 \Big( p(2p-1) a^2_k(t,y,s)-\frac{2pa_k(t,y,s)\eta_k(t)}{T-t} + \\ 
& \hspace{30mm} +\frac{2p(y_k-u_k)}{T} a_k(t,y,s)\Big) dt\Bigg\} \Bigg]\leq
 \\
& \hspace{20mm} \leq 
\exp\left\{2C_2 \|y\|+2np|2p-1|T\cdot\|a\|^2_{L_\infty(\mbR)}+2pn^{1/2}{\|u\|\cdot\|a\|_{L_\infty(\mbR)}}\right\}\times \\
& \hspace{25mm}  \times  \E\exp\left\{2p\|a\|_{L_\infty(\mbR)}\sum^n_{k=1}\int^T_0\frac{|\eta_k(t)|}{T-t}dt\right\},
\end{align*}
thus the application of Lemma \ref{lem4} finishes the proof.
\end{proof}

In parallel to $S(W+u)$ one defines $S(\eta^{u,y}), y\in\mbR^n,$ by applying the same recursive procedure  to the process $\eta^{u,y}$ and the times $\theta_{ij}(u,y), j=\ov{1,i-1},i=\ov{2,n}.$
\begin{thm}
\label{thm1} 
$\forall u\in\Delta_n \ \forall \ y\in\mbR^n \ \forall s\in Sh_n$
\[
\E\left(\1(S(W+u)=s)\cE^a_{T,n}(W, u)/W(T)=y-u\right)=\E\1(S(\eta^{u,y})=s)\mfe^a_{T,n}(u,y,s).
\]
\end{thm}
\begin{proof}
Suppose $u$ is fixed. Define, for $k=\ov{1,n},$
\[
a^m_k(t, y,s)=\sum^{m-1}_{j=0}\1\left(t\in \left(\tfrac{j}{m}T; \tfrac{j+1}{m}T\right]\right) a_{k}\left(\tfrac{j}{m}T,u,y,s\right), \quad  t\in[0; T], m\in\mbN.
\]
Note that on the set $\{S(W+u)=s\}$ for $s\in Sh_n$
\[
a(w_k(t)+u_k)\1\left( t\leq \theta_k(u)\right)=a_k(t, u, W(T)+u, s), \quad k=\ov{1, n},t\in[0; T].
\]
Due to the representation \eqref{eq2}
\begin{align}
\label{eq5}
\int^T_0a^m_k(t, w(T)+u, s)& dw_k(t)=\frac{w_k(T)}{T}\int^T_0 a^m_k(t, w(T)+u, s)dt+ \nonumber \\
& +\sum^{m-1}_{j=1}a_{k}\left(\tfrac{j}{m}T,u,y,s\right)\left(\eta_k\left(\tfrac{j+1}{m}T\right)-\eta_k\left(\tfrac{j}{m}T\right)\right).
\end{align}
Since $W(T)$ and $\eta$ are independent, the application of the disintegration theorem 
 gives, due to \eqref{eq5}, that 
\begin{align}
\label{eq6}
\E\big( \1(S(W+u)=s) \alpha_m(u) / W(T)=y-u \big)=  \E\1(S(\eta^{u,y})=s)e_m(u,y, s),
\end{align}
where
\begin{align*}
\alpha_m(u) &=   \exp\left\{\sum^n_{k=1}\int^T_0a^m_k(t, W(T)+u, s)
dw_k(t)- \right. \\
& \phantom{aaaabbbb} \left. -\frac{1}{2}\sum^n_{k=1}\int^T_0\left(a^m_k(t, W(T)+u, s)\right)^2dt
\right\}, \\
e_m(u,y, s)&=\exp\Bigg\{\sum^n_{k=1}\sum^{m-1}_{j=1}a_{k}\left(\tfrac{j}{m}T,u,y,s\right)\left(\eta_k\left(\tfrac{j+1}{m}T\right)-\eta_k\left(\tfrac{j}{m}T\right)\right)+
\\
& \phantom{aaaabbbb}
+\sum^n_{k=1}\int^T_0a^m_k(t, y, s)\left(\frac{y_k-u_k}{T}-\frac{1}{2}a^m_k(t, y, s)\right)dt\Bigg\}.
\end{align*}
Since the functions $a_k, k=\ov{1, n},$ are piecewise continuous a.s., there exists a set $\Omega'$ of full probability such that $\forall \omega\in\Omega'$ for almost all $t$ in $[0; T]$
\begin{equation}
\label{eq:tmp.a_mk}
a^m_k(t, W(T)+u, s)\mathop{\longrightarrow}\limits_{m\to\infty} a_k(t, u,W(T)+u, s) \ \mbox{in} \ \mbR.
\end{equation}

It holds due to \eqref{eq:tmp.a_mk} and the  boundedness of the function $a$ that, for $k=\ov{1,n},$
\begin{align*}
\E&\int^T_0\Bigg( \big(a_k^m(t, W(T)+u, s)-a_k(t, u,W(T)+u, s)\big)^2+ \\
&\hspace{1cm} + \big|a_k^m(t, W(T)+u, s)-a_k(t, u,W(T)+u, s)\big|\frac{|\eta_k(t)|}{T-t}\Bigg)dt\mathop{\longrightarrow}\limits_{m\to\infty} 0,
\end{align*}
hence 
\begin{equation}
\label{eq:tmp.logs}
\log e_m(u,y, s)\big\vert_{y=W(T)+u} \ \overset{\Prob}{\mathop{\longrightarrow}\limits_{m\to\infty}} \log \mfe^a_{T,n}(u,y,s)\big\vert_{y=W(T)+u}.
\end{equation}
and
\begin{equation}
\label{eq:tmp.wiener.logs}
\log \alpha_m(u) \overset{\Prob}{\mathop{\longrightarrow}\limits_{m\to\infty}} \log \cE^a_{T,n}(W, u).
\end{equation}

Repeating the reasoning of  Lemma \ref{lem5} one checks that given $y,s$ the estimate of Lemma \ref{lem5} holds for each  $e_m(u,y, s),$ so the sequence 
\[
\left\{ \1(S(\eta^{u,y})=s)e_m(u,y, s)\big\vert_{y=W(T)+u} \right\}_{m\in\mbN}
\] 
is uniformly integrable. The boundedness of the function $a$ implies the uniform integrability of the sequence $\{ \1(S(W+u)=s) \alpha_m(u)\}_{m\in\mbN}.$ Thus the claim of the theorem follows from \eqref{eq:tmp.logs}, \eqref{eq:tmp.wiener.logs} and \eqref{eq6}. 
\end{proof}

\begin{lem}\label{lem6}
$\forall u\in\Delta_n \ \forall y\in\mbR^n \ \forall s\in Sh_n$
\[
\1(S(\eta^{u,y_m})=s) \
\overset{\mbox{a.s.}}
{\mathop{\longrightarrow}\limits_{m\to\infty}}
\ \1(S(\eta^{u,y})=s),
\]
whenever $y_m\to y, m\to\infty.$
\end{lem}
\begin{proof}
We assume  that  $u$ and $y$ are fixed throughout the proof. Due to Lemma \ref{lem3} it is sufficient to check that a particular ordering of the moments $\theta_{ij}(u,y)$ is preserved in a sufficiently small random neighborhood of the point $y.$ 

The representation \eqref{eq4} implies that, with  $f(s)=\tfrac{T^2s}{Ts +1},$ 
\begin{align*}
\theta_{ij}(u,y) &= f\left(s_{ij}(y)\right),\\
s_{ij}(y)&=\inf\left\{s\ge 0\mid b_i(s)-b_j(s) + s(y_i-y_j) +\frac{u_i-u_j}{T}= 0\right\}, \\
& j=\ov{1, i-1}, i=\ov{2,n}, y\in\mbR^n,
\end{align*}
where the convention $f(\infty) =T$ is adopted.  Since $\eta^{u,y}(T)=y,$ the probability
\[
\Prob\Big(\exists k\ne j\colon  \eta_k^{u,y}(t)\neq\eta_j^{u,y}(t), t\in[0; T),  \eta_k^{u,y}(T)=\eta_j^{u,y}(T)\Big)
\]
can be greater than $0$ only if $y_j=y_k$ for some $k,j,k\neq j.$ However, in that case  
\begin{align*}
\Prob\Big(\eta_k^{u,y}(t)&\neq\eta_j^{u,y}(t), t\in[0; T),  \eta_k^{u,y}(T)=\eta_j^{u,y}(T)\Big) = \\
& = \Prob\Big( \forall s \ge 0 \ b_k(s) -b_j(s) + \frac{u_k-u_j}{T} \neq 0\Big)=0.
\end{align*}
Therefore  with probability $1$ for each $(k, j)$ either $\theta_{kj}(u,y)<T$ or
\begin{equation}
\label{eq8}
\inf_{t\in[0; T]}\left|\eta_k^{u,y}(t)-\eta_j^{u,y}(t)\right|>0.
\end{equation}
From now on, only a set of full probability which the condition \eqref{eq8} or its counterpart holds for is considered.

Fix a pair $(k, j)$ and some positive $\ve\ll1.$ Suppose $\theta_{kj}(u,y)<T.$ Obviously, there exists random $r>0$ such that the condition $\|y-y'\|<r$ implies  $\theta_{kj}(u,y')\geq\theta_{kj}(u,y)-{\ve}.$ 
The moment $\theta_{k_j}(u,y)$ is a Markov time w.r.t. the filtration generated by the process $\eta,$ therefore it follows from the iterated logarithm law for the Wiener process that there exist random $\ve_1, \ve_2: 0<\ve_1, \ve_2<\ve$ such that
\begin{align}
\label{eq9}
\eta_k(\theta_{kj}(u,y)+\ve_1)-\eta_j(\theta_{kj}(u,y)+\ve_1)&+
(u_k-u_j)\Big(1-\frac{\theta_{kj}(u,y)}{T}\Big)+(y_k-y_j)\frac{\theta_{kj}(u,y)}{T}+ \nonumber \\
&+\ve_1\Big(\frac{y_k-y_j}{T}+\frac{u_j-u_k}{T}\Big)> 0
\end{align}
and
\begin{align}
\label{eq10}
\eta_k(\theta_{kj}(u,y)+\ve_2)-\eta_j(\theta_{kj}(u,y)+\ve_2)&+
(u_k-u_j)\Big(1-\frac{\theta_{kj}(u,y)}{T}\Big)+(y_k-y_j)\frac{\theta_{kj}(u,y)}{T}+ \nonumber \\
&+\ve_2\Big(\frac{y_k-y_j}{T}+\frac{u_j-u_k}{T}\Big)< 0.
\end{align}
Thus in order to have $\theta_{kj}(u,y')\leq \theta_{kj}(u,y)+\ve$ when $\|y-y'\|\leq\delta$ it is sufficient to choose $\delta$ in such a way that 
the signs in \eqref{eq9} and \eqref{eq10} do not change when $y_k$ and $y_j$ are replaced with $y'_k$ and $y'_j,$ respectively.
\end{proof}
By using Lemmas \ref{lem5} and \ref{lem6} one can establish the following result.
\begin{lem}
\label{lem:deter.continuity}
$\forall u\in\Delta_n \ \forall s\in Sh_n$ the function $\mbR^n \ni y\mapsto \E\1(S(\eta^{u,y})=s)\mfe^{a}_{T,n}(u,y,s)$ is continuous. 
\end{lem}

\section{Finite-dimensional densities for the point process in the Arratia flow}
\label{section2}
The section is devoted to the study of the connection between the finite-dimensional densities for the Arratia flow and the stochastic exponentials that arise when addressing the Girsanov theorem for Arratia flows with drift \cite[Section 7]{1}. 
We start with the corresponding results from \cite[\S\S 7.2-7.3]{1}. Consider a dense subset of $[0; 1],$
$U=\{u_k\mid k\in\mbN\}.$ Given an Arratia flow $X$ define 
\begin{align*}
\tau_1&=T, \\
\tau_k&=\inf\Big\{s\mid \prod^{k-1}_{j=1}\left(X(u_k,s)-X(u_j,s)\right)=0\Big\} \wedge T, \quad k\geq2,
\end{align*}
and put, for $u^{(n)}=(u_1, \ldots, u_n),$
\begin{align*}
I_n\left(u^{(n)}\right)&=\sum^n_{k=1}\int^{\tau_k}_0a(X(u_k,t))dX(u_k,t), \\
J_n\left(u^{(n)}\right)&=\sum^n_{k=1}\int^{\tau_k}_0a^2(X(u_k,t))dt, \quad n\in\mbN.
\end{align*}
The integrals in the expression for $I_n$ are ordinary It\^o integrals w.r.t. the Wiener processes $X(u_j,\cdot), j\in\mbN.$ Note $ I_n$  and $J_n$ are well defined as functions on $\mbR^n.$ There exist limits
\begin{align*}
I&=L_2\mbox{-}\lim_{n\to\infty}I_n\left(u^{(u)}\right), \\
J&=L_2\mbox{-}\lim_{n\to\infty}J_n\left(u^{(u)}\right),
\end{align*}
which do not depend on the set $U.$ The distribution of an Arratia flow with drift $a$ as a random element in the Skorokhod space $D([0; 1], C([0; T]))$ is absolutely continuous w.r.t. the distribution of $X$ with density
$$
\wt\cE^a_T=\exp\left\{I-\frac{1}{2}J\right\}.
$$
Recall that $\wt\cE^a_{T,n}(u^{(n)})=\exp\{I_n(u^{(n)})-J_n(u^{(n)})\}, n\in\mbN.$

\begin{lem}
\label{lem8}
$\forall n\in\mbN$
\begin{align*}
\E\left(\wt\cE^a_{T,m}\left(u^{(m)}\right)/X(u_1,\cdot), \ldots, X(u_n,\cdot)\right)&=\wt\cE^a_{T,n}\left(u^{(n)}\right), \quad m\ge n,\\
\E\left(\wt\cE^a_T/X(u_1,\cdot), \ldots, X(u_n,\cdot)\right)&=\wt\cE^a_{T,n}\left(u^{(n)}\right).
\end{align*}
\end{lem}
\begin{proof}
The random variables $\{\wt\cE^a_{T,n}\left(u^{(n)}\right)\}_{n\in\mbN}$ form a uniformly integrable sequence (see \cite[the proof of Theorem 7.3.1, pp. 268-270]{1}), therefore it is sufficient to prove
\[
\E\left(\wt\cE^a_{T,m}\left(u^{(m)}\right)/\cG_n\right)=\wt\cE^a_{T,n}\left(u^{(n)}\right), \quad m\geq n,
\]
where $\cG_n=\sigma(X(u_1,\cdot), \ldots, X(u_n,\cdot)).$ Suppose $n$ and $m> n$ are fixed. Put 
\[
e_k(t)= \exp\Bigg\{\int^{t}_0a(X(u_k,s))dX(u_k,s)-\frac{1}{2}\int^{t}_0a^2(X(u_k,s))ds\Bigg\},\quad t\in[0;T], k\in\mbN,
\]
then, using the It\^o formula one can verify that, for $k\in\mbN,$
\begin{align*}
e_k(\tau_k)=1+\int^{\tau_k}_0 e_k(t) a(X(u_k,t))dX(u_k,t),
\end{align*}
since every $\tau_k$ is a stopping time w.r.t. the filtration generated by the processes $X(u_j,\cdot),$ $j\in\mbN.$ Therefore
\begin{align*}
\E\left(\wt\cE^a_{T,m}\left(u^{(m)}\right)/\cG_n\right)&=\wt\cE^a_{T,n}\left(u^{(n)}\right) \E\Bigg(\prod^m_{j=n+1}e_j/\cG_n\Bigg)=
\\
&=\wt\cE^a_{T,n}\left(u^{(n)}\right)\left(1+\E\left(\sum^{m-n}_{k=1}\sum_{j_1<\ldots< j_k}A_{j_1\ldots j_k}/\cG_n\right)\right),
\end{align*}
where 
\[
A_{j_1\ldots j_k}=\prod^k_{l=1}\int^{\tau_{j_l}}_0 a(X(u_{j_l}, t))e_{j_l}(t) dX(u_{j_l},t).
\]
Rewriting the multipliers in every $A_{j_1\ldots j_k}$ as
\[
\int^T_0a_{j_l}(t)dX(u_{j_l},t), \quad l=\ov{1,k},
\]
for some progressively measurable w.r.t. the filtration generated by $X(u_j,\cdot), j\in\mbN,$ processes
\[
a_{j_l}(t)=\1(t\leq\tau_{j_l})a(X(u_{j_l},t))e_{j_l}(t),
\]
applying the same approximation scheme as the one used in the proof of Theorem \ref{thm1} and utilizing the fact that the joint covariance of $X(u_i,\cdot)$ and $X(u_j,\cdot)$ equals
\[
\left(t-\inf\left\{s\mid X(u_i,s)=X(u_j,s)\right\}\right)_+, \quad t\geq0, i,j\in\mbN, 
\]
one proves via standard reasoning that
$$
\E\left(A_{j_1\ldots j_k}/\cG_n\right)=0,
$$
which concludes the proof.
\end{proof}
The following notation will be used further. 
Let $\{X^a(u, t)\mid u\in[0; 1], t\in[0; T]\}$ be an Arratia flow with drift $a.$ Define 
\begin{align*}
 \mfX_t^a(u) &=\{X^a(u_k,t)\mid k=\ov{1,n}\},  \\ 
 \vect{X}^a(u,\cdot) & \equiv \vect{X}^a(u)= \left(X^a(u_1,\cdot), \ldots, X^a(u_n,\cdot)\right), \ u=(u_1,\ldots,u_n)\in\mbR^n, n\in\mbN.
\end{align*}
Analogously to the case of coalescing Wiener processes in Section 2 one defines the coalescence scheme $S(\vect X^a(u))$ for the family $(X^a(u_1, \cdot), \ldots, X^a(u_n, \cdot)).$

Given a set $K=\{k_1, \ldots, k_m\}\subset\{1, \ldots, n\}$ and a point $z\in\mbR^n$ we denote by $z^{-K}$ the vector obtained by removing in the vector $z$ all the coordinates whose numbers are in $K;$ by $z^{K},$ the vector obtained by removing all coordinates except those in $K.$ We write $z^{K_1, \pm K_2}$ for $(z^{K_1})^{\pm K_2}.$ 

The following definitions of the finite-dimensional densities were introduced in \cite{DorVovApproximations} and represent a further development of the notions used in \cite{2,3}.

Given an Arratia flow $X^a,$ a starting point $u\in\Delta_n$ and a coalescence scheme $s\in Sh_{n,k}$ for some $k$ the corresponding $(n-j)$-point density $p^{a, n,s, n-j}_T(u; \cdot), j\geq k,$ is a measurable function on $\mbR^{n-j}$ such that for any bounded nonnegative measurable $f: \mbR^{n-j}\to\mbR$
\[
\E\sum_{
\begin{subarray}{c}v_1, \ldots, v_{n-j}\in \mfX^a_t(u),\\
v_1, \ldots, v_{n-j} \ \mbox{\small are distinct}
\end{subarray}
}
f(v_1, \ldots, v_{n-j})\1(S(\vect X^a(u))=s)=\int_{\mbR^{n-j}}p^{a,n,s,n-j}_t(u; y)f(y)dy.
\]
The existence of $p^{0,n,s,n-j}_t(u; \cdot)$ is shown in \cite[Lemma 3.1]{DorVovApproximations} whereas an explicit expression in the particular case $k=j$ is obtained in \cite[Theorem 3.1]{DorVovApproximations}.
\begin{lem}
For all $s\in Sh_{n,k},$ $u\in\Delta_n$ and $j\le n- k$ the density $p^{a,n,s,j}_t(u; \cdot)$ exists. 
\end{lem}
\begin{proof}
Let $A$ be a Borel subset of $\Delta_j.$ 
 By the Girsanov theorem for the Arratia flow 
\begin{align*}  
&\E\sum_{\begin{subarray}{c}v_1, \ldots, v_j\in \mfX_t^a(u), \\ v_1, \ldots, v_{j} \ \mbox{\small are distinct }\end{subarray}}
\1_{A}\left(v_1, \ldots, v_j\right)   \1\left(S(\vect X^a(u)) = s \right) = \\ 
  & =\E\sum_{\begin{subarray}{c}v_1, \ldots, v_j\in \mfX_t(u), \\ v_1, \ldots, v_{j} \ \mbox{\small are distinct }\end{subarray}}
\1_{A}\left(v_1, \ldots, v_j\right)   \1\left(S(\vect X(u)) = s \right) \wt\cE^a_{T,n} \left(u\right)  \le  \\
&\hspace{1cm} \le \sum_{
\begin{subarray}{c}
L=\{l_1,\ldots, l_j\},\\
l_i\in\{1,\dots, n-k\}, \ i=\ov{1,j}
\end{subarray}
} \E \1_{A}\left( \vect X(u^L,t)\right)    \E\left( \wt\cE^a_{T,n} \left(u\right) / \vect X(u^L,t)\right) \leq \\ 
&\hspace{1cm} \le \sum_{
\begin{subarray}{c}
L=\{l_1,\ldots, l_j\},\\
l_i\in\{1,\dots, n-k\}, \ i=\ov{1,j}
\end{subarray}
} \int_{A}p^{0,j,j}_{t}(u^L; y)\E\left( \wt\cE^a_{T,n} \left(u\right) / \vect X(u^L,t) =y\right)
  dy.  
\end{align*}
 The case when $A$ is not a subset of $\Delta_j$ is treated similarly.
\end{proof}

Fix some $u\in\mbR^n$ and $k\in\{0,\ldots, n-1\}.$ Following \cite[pp. 433-434]{dorogovtsev.2011.multiplicative}, we associate with a coalescence scheme $s=(j_1,\ldots,j_k)$ a partition of the set $\{1,\ldots, n\}$ as follows. Starting from the partition consisting of singletons, at each step $i =1,\ldots, k$ proceed by merging two subsequent blocks in the current partition with the numbers $j_i$ and $j_i+1,$ the blocks being listed in ascending order of their minimal elements. Let the blocks of the final partition be  $\pi_1, \ldots, \pi_k.$ We define the set $I(s)=\{ \min \pi_i\mid i=\ov{1,n-k}\}.$ As a result, 
\[
|\{X^a(u_i, T)\mid i\in I(s)\}|=n-k
\]
on $\{S(\vect X^a(u)) = s\}.$ Obviously, the coalescence scheme deterministically and uniquely defines $I(s),$ which does not depend on $u$ and a specific realization of the flow $X^a.$

Denote by $g^m_T(u; \cdot)$  the $m$-dimensional Gaussian density with mean $u$ and variance $T\Id_{m\times m},$ where $\Id_{m\times m}$ is the unit square matrix of size $m,$ $m\in\mbN.$

\begin{thm}[cf. \cite{kon.marx}]
\label{thm2}
Assume $u\in\Delta_n$ and $s\in Sh_{n, n-k}$ for some $k\in\{0, \ldots, n-1\}.$ Then for each $j\in\{1, \ldots, k\}$ for  all $y\in\Delta_k$
\begin{align*}
p^{a, n,s, j}_T(u; y)=\sum_{
\begin{subarray}{c}L=\{l_1,\ldots, l_j\}\subset\\
\{1, \ldots, k\}\end{subarray}}
g^j_T\left(u^{I(s), L}; z^{I(s), L}\right)
\int_{\mbR^{k-j}}dz^{I(s), -L}g^{k-j}_T\left(u^{I(s), -L}; z^{I(s), -L}\right) 
 \\ 
\int_{\mbR^{n-k}}dz^{-I(s)}g^{n-k}_T(u^{-I(s)}; z^{-I(s)}) \left(\E\1(S(\eta^{u,z})=s)\mfe^a_{T,n}(u,z, s)\right)\Big|_
{\begin{subarray}{l}z\in\mbR^n,\\ z^{I(s),L}=y\end{subarray}
}.
\end{align*}
\end{thm}
\begin{proof}
Fix $u$ and take $W$ and $\wt W$ as in Section 2. Define
\begin{align*}
B^+_\delta(v)&=[v; v+\delta), \quad v\in\mbR,\  \delta>0, \\
\xi&=\wt{W}(T)+u, \\
\mathcal{W} &= \left\{ \wt{w}_1(T)+u_1, \ldots, \wt{w}_n(T)+u_n \right\}.
\end{align*}

We have, due to the Lebesgue differentiation theorem and \cite[Theorem 7.3.1]{1}, for almost all $y$
\begin{align*}
p^{a,n,s,j}_T(u; y)&=\lim_{\delta\to0+}\delta^{-j}\E
\prod^j_{i=1}\left| B^+_\delta(y_i) \cap \mfX^a_T(u) \right|\cdot\1(S(\vect X^a(u))=s)=
\\
&=\lim_{\delta\to0+}\delta^{-j}\E
\prod^j_{i=1}\left| B^+_\delta(y_i) \cap \mathcal{W} \right|\cdot \cE^a_{T, n}(W,u) \1(S(W+u)=s),
\end{align*}
since by Lemma \ref{lem1} $\vect X(u)\overset{d}{=}\wt{W}+u.$  
The reasoning of \cite[Appendix B]{2}, combined with Lemma \ref{lem5}, allows one to replace $| B^+_\delta(y_i) \cap \mathcal{W}|$ with $\1(B^+_\delta(y_i) \cap \mathcal{W} \ne \emptyset)$ for all $i$ in the previous formula, so that 
\begin{multline*}
p^{a,n,s,j}_T(u; y)= \\ = \lim_{\delta\to0+}\delta^{-j}\E\sum_
{L=\{l_1,\ldots, l_j \}\subset\{1, \ldots,k\}}
\prod^j_{i=1}
\1_{B^+_\delta(y_i)}\left(\xi_{l_i}\right) \cE^a_{T,n}(W,u) \1(S(W+u)=s).
\end{multline*}
Consider a separate term 
\[
A_{\delta,L}=\E\prod^j_{i=1}\1_{B^+_\delta(y_i)}\left(\xi_{l_i}\right)\cE^a_{T,n}(W,u) \1(S(W+u)=s).
\]
Since on the set $\{S(W+u)=s\}$ by the definition of $I(s)$
\[
\mathcal{W}=\{w_i(T)+u_i\mid i\in I(s)\},
\]
Theorem \ref{thm1} implies that
\begin{align*}
A_{\delta,L}&=\E\E\left(\prod^j_{i=1}\1_{B^+_\delta(y_i)}\left(\xi_{l_i}\right)\1(S(W+u)=s)\cE^a_{T,n}(W,u)/W(T)\right)=
\\
&=\E\prod^j_{i=1}\1_{B^+_\delta(y_i)}\left((W(T)+u)^{I(s),L}_i\right)  \left(\E\1(S(\eta^{u,z})=s)\mfe^a_{T,n}(u,z, s)\right)\Big|_{z=W(T)+u}.
\end{align*}
The vectors $W(T)^{I(s),L}, W(T)^{I(s), -L}$ and $W(T)^{-I(s)}$ being independent given fixed $s,$ the local property of conditional expectation \cite[Lemma 6.2]{7} implies that
\begin{align*}
p^{a,n,s,j}_T(u; y)&=\lim_{\delta\to0+}
\delta^{-j}
\E\sum_{L=\{l_1,\ldots, l_j\}\subset
\{1, \ldots, k\}}
\prod^j_{i=1}\1_{B^+_\delta(y_i)}\left((W(T)+u)^{I(s),L}_i\right)\times \\
& \phantom{aaaa} \times \left(\E\1(S(\eta^{u,z})=s)\mfe^a_{T,n}(u,z, s)\right)\Bigg|_{
\begin{subarray}{l}
z^{I(s),L}=(W(T)+u)^{I(s),L},\\
z^{I(s), -L}=\alpha,\\
z^{-I(s)}=\beta,
\end{subarray}}
\end{align*}
where  $W(T)^{I(s),L}$ and the Gaussian random variables $\alpha\sim\cN(u^{I(s),-L}, T\cdot \mathrm{Id}_{(k-j)\times(k-j)}),$ 
$\beta\sim\cN(u^{-I(s)}, T \mathrm{Id}_{(n-k)\times(n-k)})$
are jointly independent. With Lemma \ref{lem:deter.continuity}, the rest of the proof follows by standard reasoning. We omit the details.
\end{proof}

 Given an Arratia flow $X^a,$ a starting point $u\in\Delta_n$ and $k\in\{1, \ldots, n\}$ the corresponding $(n,k)$-point density is a measurable function $p^{a,n,k}_T(u; \cdot)$ on $\mbR^k$ such that for any bounded nonnegative measurable $f: \mbR^k\to\mbR$
 \begin{equation}
 \label{eq11}
 \E\sum_{
\begin{subarray}{c}v_1, \ldots, v_{k}\in \mfX^a_t(u),\\
v_1, \ldots, v_{k} \ \mbox{\small are distinct}
\end{subarray}
}
f(v_1, \ldots, v_{k})\1\left(|\mfX^a_t(u)|\geq k\right)=\int_{\mbR^{k}}p^{a,n,k}_T(u; y)f(y)dy.
 \end{equation}

The next consequence of the formula of total probability  gives a relation between $p^{a,n,k}_T$ and $p^{a,n,s,k}_T.$

\begin{lem}
\label{lem7}
 For any $n\in\mbN, u\in\Delta_n$ and $k\in\{1,\ldots, n\}$ a.e.
\[
p^{a,n,k}_T(u; \cdot)=\sum^{n-k}_{l=0}\sum_{s\in Sh_{n,l}}p^{a,n,s,k}_T(u; \cdot).
\]
\end{lem}

The $k-$point density $p^{a,k}_T(\cdot)$ (cf \cite{2, 3}) is defined as a measurable function on $\mbR^k$ such that the analog of \eqref{eq11} holds with $\mfX^a_T(u)$ replaced with the set $\{X^a(v,T)\mid v\in[0; 1]\}$ and the condition $|\mfX^a_t(u)|\geq k$ dropped. The result of \cite[Theorem 3.2]{DorVovApproximations} admits the following extension, the proof being the same with some minor changes having been made.

\begin{thm}
\label{thm3}
Let $u^{(n)}=(u^{(n)}_1, \ldots, u^{(n)}_n)\in\Delta_n, n\in\mbN,$ be such that $u^{(n)}_1=0, u^{(n)}_n=1, n\in\mbN,$
\[
\big\{u^{(n)}_1, \ldots, u^{(n)}_n\big\}\subset \Big\{u^{(n + 1)}_1, \ldots, u^{(n + 1)}_{n + 1}\Big\}, 
\quad n\in\mbN,
\]
 and
\[
\max_{j=\ov{0, n-1}}\left(u^{(n)}_{j+1}-u^{(n)}_j\right)
{\mathop{\longrightarrow}\limits_{n\to\infty}} 0.
\]
Then for all $k\in\mbN$ a.e.
\[
p^{a,n,k}_T\left(u^{(n)}; \cdot\right)\nearrow p^{a,k}_T, \ n\to\infty.
\]
\end{thm}

Due to Lemma \ref{lem7}, Theorem \ref{thm2} provides an explicit expression for the densities $p^{a,k}_T(u;\cdot)$ in terms of conditional expectations of certain stochastic exponentials discussed in Section \ref{section1}.

Consider $u\in\Delta_n.$ Suppose that elements of the set $\mfX_T(u)$ are listed in ascending order. Let $\vk$ be the cemetery state. Given a set  $L=\{l_1, \ldots, l_k\}, l_i\in\mbN, i=\ov{1,k},$ for some $k,$ put the random vector $\mcX^L_T(u)$ to be equal to
\begin{equation}
\label{eq12}
\left((\mfX_T(u))_{l_1,\ldots,}(\mfX_T(u))_{l_k}\right),
\end{equation}
if $\max_{i=\ov{1,k}}l_i\leq|\mfX_T(u)|,$
and $\vk,$ otherwise. We denote the density of $\mcX^L_T(u)$ in $\mbR^k$ by $q^L_T(u;\cdot).$ Note that always
\[
\int_{\mbR^k}q^L_T(u;y)dy<1.
\]

\begin{thm}[cf. \cite{kon.marx}]
\label{thm4}
For any $u\in\Delta_n$ and any $k\in\{1, \ldots, n\}$ a.e.
\[
p^{a,n,k}_T(u;y)=\sum_{
L=\{l_1,\ldots, l_k\},\ 
l_i\in\mbN,  \ i=\ov{1,k}}
q^L_T(u; y)\E\left(\wt\cE^a_{T,n}(u)/\mcX^L_T(u)=y\right).
\]
\end{thm}

\begin{proof}
Proceeding similarly to the proof of Theorem \ref{thm2}, one obtains a.e.
\begin{align*}
p^{a,n,k}_T(u;y)&=
\lim_{\delta\to0+}\delta^{-k}
\sum_{
\begin{subarray}{c}
L=\{l_1,\ldots, l_k\},\\
l_i\in\mbN, \ i=\ov{1,k}
\end{subarray}
}
\E\wt\cE^a_{T}(u)\1\left(\mcX^L_T(u)\ne\vk\right)\prod^k_{i=1}\1_{B^+_\delta(y_i)}\left((\mcX^L_T(u))_i\right)= \\
&=
\lim_{\delta\to0+}\delta^{-k}
\sum_{
\begin{subarray}{c}
L=\{l_1,\ldots, l_k\},\\
l_i\in\mbN, \ i=\ov{1,k}
\end{subarray}
}
\E\1\left(\mcX^L_T(u)\ne\vk\right)\times\1\left(\mcX^L_T(u)\in\mathop{\times}\limits^k_{i=1}B^+_\delta(y_i)\right)\times \\
& \phantom{aaaa} \times \E\left(\wt\cE^a_T(u)/\mcX^L_T(u)\right).
\end{align*}
The application of the Lebesgue differentiation theorem finishes the proof. 
\end{proof}
Replacing in \eqref{eq12} the set $\mfX_T(u)$ with the set $\{X(v,T)\mid v\in[0; 1]\}$ one defines, analogously to $\mcX^L_T(u),$ the random vector $\mcX^L_T$ with values in $\mbR^k\cup\{\vk\}.$ The corresponding density being denoted by $q^L_T(\cdot),$ the following result holds.
\begin{thm}
\label{thm5}
For any $k\in\mbN$ a.e.
$$
p^{a,k}_T(y)=
\sum_{
L=\{l_1,\ldots, l_k\}, \ l_i\in\mbN, \ i=\ov{1,k}}
q^L_T(y)\E\left(\wt\cE^a_T/\mcX^L_T=y\right).
$$
\end{thm}
The arguments 
 are repetitive and thus omitted.

\end{document}